\documentclass{article}
\usepackage[dvips]{graphicx}

\usepackage{amssymb,amsthm, amsmath,enumerate}
\usepackage{color,graphicx,url}

\def\Z{{\mathbf{Z}}}    
\def\R{{\mathbf{R}}}    
\def\C{{\mathbf{C}}}    

\def\0{{\mathbf{0}}}

\def\#{\sharp}

\def\Spec{{\mathrm{Spec}}}

\renewcommand{\Re}{\mathrm{Re}}

\def\Arg{{\mathrm{Arg}}}

\def\cP{\mathcal{P}}
\def\cZ{\mathcal{Z}}

\def\Ei{{\mathrm{Ei}}}

\newcommand{\Res}[2]
{\,\underset{#1=#2}{\mathrm{Res}}\,}

\newtheorem{theorem}{Theorem}
\newtheorem*{maintheorem*}{Main Theorem}
\newtheorem*{theorem*}{Main Theorem}
\newtheorem{corollary}{Corollary}

\newtheorem{lemma}{Lemma}
\newtheorem*{keylemma*}{KeyLemma}
\newtheorem{prop}{Proposition}
\newtheorem{example}{Example}

\newtheorem*{fact*}{Fact}

\newtheorem*{remark*}{Remark}
\renewenvironment{proof}%
	{\smallbreak\noindent {\bf Proof.\ }}%
	{\hskip 1em \rule[-2pt]{5pt}{10pt}

 \medbreak}
\title{On graph theory Mertens' theorems}
\author{
Takehiro HASEGAWA (up date 01/Mar/2014) \\
Shiga University, Otsu, Shiga 520-0862, Japan 
\\
Seiken SAITO \\
Waseda University, Shinjuku, Tokyo 169-8050, Japan
\\
}

\begin{document}
\maketitle

\begin{abstract} 
In this paper, we study graph-theoretic analogies of the Mertens' theorems by using basic properties of the Ihara zeta-function. 
One of our results is a refinement of a special case of the dynamical system Mertens' second theorem due to Sharp and Pollicott. 
\end{abstract}

\bigskip


\section{Introduction}\label{sect1}

Throughout this paper, we use the notation in the textbook \cite{Terras2011} of Terras for graph theory and the (Ihara) zeta-function, and we often refer to basic facts included in this textbook. 

\bigskip

In 1874, Mertens proved the so-called Mertens' first/second/third theorems (the equalities (5)(13)(14) in \cite{Mertens}, respectively). 
In 1991, Sharp studied the dynamical-systemic analogues of Mertens' second/third theorems (Theorem 1 in \cite{Sharp}), and in 1992, Pollicott improved the error terms in the theorems of Sharp as follows (Theorem and Remark in \cite{Pollicott}): 
\begin{itemize}
\item 
Dynamical system Mertens' second theorem: \ 
For a hyperbolic (and so geodesic) flow (which is not necessarily topologically weak-mixing) restricted to a basic set with closed orbits $\tau$ of least period $\lambda(\tau)$ and topological entropy $h>0$, as $x \to \infty$, 
\begin{align*}
\sum_{N(\tau) \leq x} \frac{1}{N(\tau)} = \log \left( \log x \right) &+ \gamma+ \log \left( \Res{s}{1} \zeta(s) \right) \\
&- \sum_{\tau} \sum_{n \geq 2} \frac{1}{n} \cdot \frac{1}{N(\tau)^{n}} + O \left( \frac{1}{\log x} \right), 
\end{align*}
where $N(\tau)= e^{h \lambda(\tau)}$, $\gamma$ is the Euler-Mascheroni constant, 
$$
\zeta(s)= \prod_{\tau} \left( 1- \frac{1}{N(\tau)^{s}} \right)^{-1}, 
$$
and $\Res{s}{1} \zeta(s)$ denotes the residue of $\zeta$ at $s=1$. 

\item 
Dynamical system Mertens' third theorem: \ 
For the same flow, as $x \to \infty$, 
\begin{align*}
\prod_{N(\tau) \leq x} \left( 1- \frac{1}{N(\tau)} \right)= \frac{e^{-\gamma}}{\Res{s}{1} \zeta(s)} \cdot \frac{1}{\log x} \left( 1+ O \left( \frac{1}{\log x} \right) \right). 
\end{align*}
\end{itemize}
(For the notation of dynamical systems, see the textbook \cite{Parry-Pollicott1990} of Parry-Pollicott.)
Note that Sharp and Pollicott did not explicitly write the dynamical system Mertens' first theorem. 

\medskip

From the second theorem of Pollicott, the constant term (so-called Mertens constant) can be explicitly known, but the coefficients of $1/\log^{k} x$ can not be computed. 
Our purpose in this paper is to present graph-theoretic analogies of the Mertens' theorems whose coefficients can be explicitly known. 
So, our second theorem is a refinement of a special case of the theorem due to Pollicott in the sense that the coefficients of $1/\log^{k} x$ can be computed. 

\bigskip

In the rest of this section, we introduce the notation of graph theory, next recall the notation and properties of the (Ihara) zeta-function, and last state the main theorem. 

\medskip

First, we recall the notation of graphs. 
Let $X$ be an undirected graph with vertex set $V$ of $\nu:= |V|$ and edge set $E$ of $\epsilon:= |E|$. 
Simply, such a graph $X$ is denoted by $X:=(V, E)$. 

A directed edge (or an arc) $a$ from a vertex $u$ to a vertex $v$ is denoted by $a= (u,v)$, and the inverse of $a$ is denoted by $a^{-1}= (v,u)$. 
The origin (resp. terminus) of $a$ is denoted by $o(a):= u$ (resp. $t(a):= v$). 

We can direct the edges of $X$, and label the edges as follows: 
$$
\vec{E}:= \big\{ e_{1}, \quad e_{2}, \quad \ldots, \quad e_{\epsilon}, \quad e_{\epsilon+1}= e_{1}^{-1}, \quad e_{\epsilon+ 2}= e_{2}^{-1} \quad \ldots, \quad e_{2\epsilon}= e_{\epsilon}^{-1} \big\}. 
$$
A path $C=a_{1} \cdots a_{s}$, where the $a_{i}$ are directed edges, is said to have a backtrack (resp. tail) if $a_{j+1}= a_{j}^{-1}$ for some $j$ (resp. $a_{s}= a_{1}^{-1}$), and a path $C$ is called a cycle (or closed path) if $o(a_{1})= t(a_{s})$. 
The length $\ell(C)$ of a path $C= a_{1} \cdots a_{s}$ is defined by $\ell(C):= s$. 

A cycle $C$ is called prime (or primitive) if it satisfies the following: 
\begin{itemize}
\item
$C$ does not have backtracks and a tail; 
\item
no cycle $D$ exists such that $C= D^f$ for some $f>1$. 
\end{itemize}
The equivalence class $[C]$ of a cycle $C= a_{1} \cdots a_{s}$ is defined as the set of cycles 
$$
[C]:= \big\{ a_{1} a_{2} \cdots a_{s-1} a_{s}, \qquad a_{2} \cdots a_{s-1} a_{s} a_{1}, \qquad \dots, \qquad a_{s} a_{1} a_{2} \cdots a_{s-1} \big\}, 
$$
and an equivalence class $[P]$ of a prime cycle $P$ is called prime. 

Let $\Delta_{X}$ and $\pi_{X}(n)$ denote 
\begin{align*}
\Delta &= \Delta_{X}:= \gcd \{ \ell(P) : \text{$[P]$ is a prime equivalence class in $X$} \}, \\
\pi(n) &= \pi_{X}(n):= | \{ [P] : \textrm{$[P]$ is a prime equivalence class in $X$ with $\ell(P)=n$} \}|. 
\end{align*}


Throughout this paper, we always assume that $X$ is a finite, connected, noncycle and undirected graph without degree-one vertices, and we denote by a symbol $[P]$ a prime equivalence class. 

\medskip

Next, we recall the zeta-function of $X= (V,E)$. 
The (Ihara) zeta-function of $X$ is defined as follows (the equality (9) in \cite{Ihara}, and also see Definition 2.2 in \cite{Terras2011}): 
\begin{align*}
Z_{X}(u) :&= \prod_{[P]}(1-u^{\ell(P)})^{-1} 
\end{align*}
with $|u|$ sufficiently small, where $[P]$ runs through all prime equivalence classes in $X$. 
The radius of convergence of $Z_{X}(u)$ is denoted by $R_{X}$. 
Note that $0< R_{X}<1$ since $X$ is a noncycle graph (see, for example, page 197 in \cite{Terras2011}). 

Let $W= W_{X}:= (w_{ij})$ denote the edge adjacency matrix of a graph $X$, that is, a $2 \epsilon \times 2 \epsilon$ matrix defined by 
$$
w_{ij} := 
\begin{cases}
\ 1 & \text{if $t(e_{i})= o(e_{j})$ and $e_{j} \neq e_{i}^{-1}$ for $e_{i}, e_{j} \in \vec{E}$}, \\
\ 0 & \text{otherwise} 
\end{cases}
$$
(see page 28 in \cite{Terras2011}). 

\bigskip

In this paper, our main theorem is: 
\begin{maintheorem*}
Suppose that $X= (V, E)$ is a finite, connected, noncycle and undirected graph without degree-one vertices. 
Set 
$$
a= a(N):= \left\{ \frac{N}{\Delta_{X}} \right\} \Delta_{X} \quad \left(= N- \left[ \frac{N}{\Delta_{X}} \right] \Delta_{X} \right), 
$$
where $[x]$ (resp. $\{x \}$) denotes the integer (resp. fractional) part of the real number $x$, and thus $0 \leq a(N)< \Delta_{X}$. 
Then, the following items (1)(2)(3) hold: 
\begin{enumerate}
\item[(1)] \ 
(Graph theory Mertens' first Theorem) \ 
As $N \to \infty$, 
\begin{align*}
\sum_{n \leq N} n \cdot \pi_{X}(n) R_{X}^{n} &= N- a(N)+ A_{X}+ K_{X}+ O \left( (\rho_{X} R_{X})^{N} \right) \\
& \left( = \left[ \frac{N}{\Delta_{X}} \right] \Delta_{X}+ A_{X}+ K_{X}+ O \left( (\rho_{X} R_{X})^{N} \right) \right), 
\end{align*}
where the constants $A_{X}$, $K_{X}$ and $\rho_{X}$ are defined by 
$$
A_{X}:= \sum_{\lambda \in \text{Spec}(W), \atop |\lambda|< 1/R_{X}} \frac{\lambda R_{X}}{1- \lambda R_{X}}, \qquad K_{X}:= \sum_{n=1}^{\infty} \left( n \cdot \pi_{X}(n)- \sum_{\lambda \in \Spec(W)} \lambda^{n} \right) R_{X}^{n}, 
$$
and 
$$
\rho_{X}:= \max \{ |\lambda| : \lambda \in \text{Spec}(W), \ |\lambda|< 1/R_{X} \}, 
$$
respectively. 
(The convergence of $K_{X}$ is shown in Section \ref{sect2}.)

\item[(2)] \ 
(Graph theory Mertens' second Theorem) \ 
As $N \to \infty$, 
\begin{align*}
\sum_{n \leq N} \pi_{X}(n) R_{X}^{n} = \log N &+ \gamma+ \log C_{X}- H_{X} \\
& - \sum_{s=1}^{k} \left( \frac{a^{s}}{s}+ \sum_{m=0}^{s-1} \binom{s-1}{m} \frac{a^{m} B_{s-m} \Delta_{X}^{s-m}}{s-m} \right) \frac{1}{N^{s}}+ O \left( \frac{1}{N^{k+1}} \right) 
\end{align*}
for each $k \geq 1$, where $\gamma$ is the Euler-Mascheroni constant, $B_{s}$ are the $s$-th Bernoulli numbers defined by 
$$
\frac{t}{e^{t}-1}= \sum_{s=0}^{\infty} B_{s} \frac{t^{s}}{s!}, 
$$
and the constants $C_{X}$ and $H_{X}$ are defined by 
$$
C_{X}:= - \frac{1}{R_{X}} \Res{u}{R_{X}} Z_{X}(u), \qquad H_{X}:= - \sum_{n \geq 1} \frac{1}{n} \left( n \cdot \pi_{X}(n)- \sum_{\lambda \in \text{Spec}(W)} \lambda^{n} \right) R_{X}^{n}, 
$$
respectively. 
(The convergence of $H_{X}$ is shown in Section \ref{sect2}.) 

In particular ($k=0$), as $N \to \infty$, 
\begin{align*}
\sum_{n \leq N} \pi_{X}(n) R_{X}^{n} = \log N + \gamma+ \log C_{X}- H_{X} + O \left( \frac{1}{N} \right). 
\end{align*}

\item[(3)] \ 
(Graph theory Mertens' third Theorem, \cite{Hasegawa-Saito}) \ 
As $N \to \infty$, 
\begin{align*}
\prod_{n \leq N} (1- R_{X}^{n})^{\pi_{X}(n)}= \frac{e^{-\gamma}}{C_{X}} \cdot \frac{1}{N} \left( 1+ O \left( \frac{1}{N} \right) \right). 
\end{align*}
\end{enumerate}
\end{maintheorem*}

Note that our first theorem (1) is a refinement of a result in our previous paper \cite{Hasegawa-Saito} in the sense that the constant term $A_{X}+ K_{X}$ is explicitly written, and note that our second theorem (2) is a refinement of a special case of the result due to Pollicott (Theorem (i) and Remark in \cite{Pollicott}) in the sense that all the coefficients of $1/N^{k}$ can be explicitly computed. 
Our proofs in this paper are elementary (without the theory of the Ihara prime zeta-function which is studied in \cite{Hasegawa-Saito}), and moreover they are completely different from previous proofs. 

\bigskip

Our theorems (1)(2) can be simplified under the assumption which all the degrees of vertices are greater than $2$: 
If $X$ is bipartite, then $\Delta_{X}=2$, and so $a(2N)=0$ or $a(2N+1)=1$. 
Otherwise, $\Delta_{X}=1$, and therefore $a(N)=0$. 
(See, for details, Proposition 3.2 in \cite{Kotani-Sunada}.)

\bigskip

The contents of this paper are as follows. 
In the next section, we first prove a keylemma, which plays an important role in the proof of the main theorem, and next introduce the constants in the main theorem. 
In Section \ref{sect3}, we give the proof of the main theorem. 

\section{KeyLemma}\label{sect2}

In this section, in order to show the theorem, we introduce a keylemma and two constants. 

\bigskip

The following facts are often used in this paper. 

\begin{fact*}
Suppose that $X= (V, E)$ satisfies the same conditions as the main theorem. 

\medskip

(1) \ (Theorem 1.4 in \cite{Kotani-Sunada}, and also see Theorem 8.1 (3) in \cite{Terras2011}) \ 
The poles of $Z_{X}(u)$ on the circle $|u|= R_{X}$ have the form $R_{X} e^{2 \pi i a/\Delta_{X}}$, where $a= 1, 2, \ldots, \Delta_{X}$. 

\medskip

(2) \ (Orthogonality relation, for example, see Exercise 10.1 in \cite{Terras2011}) \ 
$$
\sum_{a=1}^{\Delta_{X}} e^{2 \pi i an/\Delta_{X}}= 
\begin{cases}
\ \Delta_{X} & \text{if $\Delta_{X} \mid n$}; \\
\ 0          & \text{otherwise}. 
\end{cases}
$$

\medskip

(3) \ (Two-term determinant formula, \cite{Hashimoto} and \cite{Bass}, and also see (4.4) in \cite{Terras2011}) \ 
The zeta-function $Z_{X}(u)$ can be written as 
$$
Z_{X}(u)= 1/\det (I_{2 \epsilon}- W u)= \prod_{\lambda \in \text{Spec}(W)} (1- \lambda u)^{-1}. 
$$
\end{fact*}

\bigskip

The following keylemma plays an important role in the proof of the main theorem.
\begin{keylemma*}
Suppose that $X= (V, E)$ satisfies the same conditions as the main theorem. 

\medskip

(a) \ As $N \to \infty$, 
$$
\sum_{n= 1}^{N} \sum_{\lambda \in \text{Spec}(W)} (\lambda R_{X})^{n}= \left[ \frac{N}{\Delta_{X}} \right] \Delta_{X}+ A_{X}+ O \left( (\rho_{X} R_{X})^{N} \right), 
$$
where $[x]$ denotes the integer part of the real number $x$.

\medskip

(b) \ As $N \to \infty$, 
$$
\sum_{n= 1}^{N} \frac{1}{n} \sum_{\lambda \in \text{Spec}(W)} (\lambda R_{X})^{n}= \sum_{n=1}^{[ {N}/{\Delta_{X}} ]} \frac{1}{n}+ \log C_{X}+ \log \Delta_{X}+ O \left( (\rho_{X} R_{X})^{N} \right). 
$$

\medskip

(c) \ (\cite{Hasegawa-Saito}) \ Let $0< \alpha< 1/2$ be a real number, and fix it. 
Then, there exists a natural number $N_{0}= N_{0}(\alpha)$ such that for any $n \geq N_{0}$, 
$$
\left| n \cdot \pi_{X}(n)- \sum_{\lambda \in \text{Spec}(W)} \lambda^{n} \right|< 2 \epsilon \left( \frac{1}{R_{X}} \right)^{(1- \alpha)n}. 
$$

\medskip

(d) \ (cf. Section 2 in \cite{Knuth}) \ 
Set 
$$
a= a(N):= N- \left[ \frac{N}{\Delta_{X}} \right] \Delta_{X}, 
$$
and thus $0 \leq a(N)< 1$. 
Then, 
\begin{align*}
\sum_{n=1}^{[N/\Delta_{X}]} \frac{1}{n} = \log N &- \log \Delta_{X}+ \gamma \\
& - \sum_{s=1}^{k} \left( \frac{a^{s}}{s}+ \sum_{m=0}^{s-1} \binom{s-1}{m} \frac{a^{m} B_{s-m} \Delta_{X}^{s-m}}{s-m} \right) \frac{1}{N^{s}}+ O \left( \frac{1}{N^{k+1}} \right) 
\end{align*}
for each $k \geq 1$. 
\end{keylemma*}
\begin{proof}
In this proof, we abbreviate the suffix $X$, that is, $R= R_{X}$, $\Delta= \Delta_{X}$, etc. 

Let $\{ a_{n} \}$ be a sequence of real numbers with $0< a_{n} \leq 1$ for any $n$. 
Then, it follows from Facts (1)(2)(3) that we obtain the equality 
\begin{align}\label{fact123}
\sum_{n=1}^{N} a_{n} \sum_{\lambda \in \text{Spec}(W), \atop |\lambda|= 1/R} (\lambda R)^{n}= \sum_{n=1}^{N} a_{n} \sum_{a=1}^{\Delta} e^{-2 \pi i an/\Delta}= \Delta \sum_{n=1}^{[N/\Delta]} a_{n \Delta}. 
\end{align}
Moreover, it follows by the triangle inequality that we obtain the inequality 
\begin{align}\label{otherfact123}
\left| \sum_{n>N} a_{n} \sum_{\lambda \in \text{Spec}(W), \atop |\lambda|< 1/R} (\lambda R)^{n} \right| & \leq \sum_{n>N} a_{n} \sum_{\lambda \in \text{Spec}(W), \atop |\lambda|< 1/R} (|\lambda| R)^{n} \nonumber \\
& < 2 \epsilon \sum_{n>N} (\rho R)^{n}= \frac{2 \epsilon \rho R}{1- \rho R} (\rho R)^{N}. 
\end{align}
In the proofs of the items (a)(b), we use the equality (1) and the inequality (2). 

\medskip

(a) \ Note that 
\begin{align*}
\left| \sum_{n=1}^{N} \sum_{|\lambda|= 1/R} (\lambda R)^{n}- \left[ \frac{N}{\Delta} \right] \Delta \right| &= \left| \sum_{n=1}^{[N/\Delta]} \Delta- \left[ \frac{N}{\Delta} \right] \Delta \right|= 0 
\end{align*}
by the equality (\ref{fact123}). 
On the other hand, note that 
\begin{align*}
\left| \sum_{n=1}^{N}  \sum_{|\lambda|< 1/R} (\lambda R)^{n}- A \right| &= \left| \sum_{n> N} \sum_{|\lambda|< 1/R} (\lambda R)^{n} \right| < \frac{2 \epsilon \rho R}{1- \rho R} (\rho R)^{N} 
\end{align*}
from the inequality (\ref{otherfact123}). 
By combining these, the item (a) follows from the triangle inequality. 

\medskip

(b) \ Set the sums 
$$
S_{1}(N):= \sum_{n= 1}^{N} \frac{1}{n} \sum_{|\lambda|= 1/R} (\lambda R)^{n} \qquad \text{and} \qquad S_{2}(N):= \sum_{n= 1}^{N} \frac{1}{n} \sum_{|\lambda|< 1/R} (\lambda R)^{n}. 
$$

First, we consider the sum $S_{1}(N)$. 
It follows from the equality (\ref{fact123}) that 
\begin{align*}
S_{1}(N) &= \sum_{n=1}^{N} \frac{1}{n} \sum_{|\lambda|= 1/R} (\lambda R)^{n} = \sum_{n=1}^{[N/\Delta]} \frac{1}{n}. 
\end{align*}

Next, we compute the sum $S_{2}(N)$. 
We now consider the constant defined by 
\begin{align*}
F:= \log \prod_{|\lambda|< 1/R} \frac{1}{1- \lambda R} \quad \left( = - \sum_{|\lambda|< 1/R} \log (1- \lambda R)= \sum_{|\lambda|< 1/R} \sum_{n \geq 1} \frac{1}{n} (\lambda R)^{n} \right), 
\end{align*}
and then we obtain $F= \log C_{X}+ \log \Delta$. 
This is proved as follows: \ 
Note that 
\begin{align*}
C_{X}= \lim_{u \uparrow R} \frac{(R-u) Z_{X}(u)}{R} &= \lim_{u \uparrow R} \left( 1- \frac{u}{R} \right) \prod_{\lambda \in \text{Spec}(W)} \frac{1}{1- \lambda u} \\
&= \lim_{u \uparrow R} \prod_{\lambda \in \text{Spec}(W), \atop \lambda \neq 1/R} \frac{1}{1- \lambda u}= \prod_{\lambda \in \text{Spec}(W), \atop \lambda \neq 1/R} \frac{1}{1- \lambda R} 
\end{align*}
by the definition of $C_{X}$ and Facts (1)(3). 
It is well known that 
\begin{align*}
\sum_{n=0}^{\Delta-1} X^{n}= \frac{X^{\Delta}-1}{X-1}= \prod_{a=1}^{\Delta-1} \left( X- e^{-2 \pi i a/\Delta} \right), \quad \text{and so} \quad \Delta= \prod_{a=1}^{\Delta-1} \left( 1- e^{-2 \pi i a/\Delta} \right). 
\end{align*}
Combining these equalities, we obtain 
\begin{align*}
\prod_{|\lambda|< 1/R} \frac{1}{1- \lambda R} &= \prod_{\lambda \in \text{Spec}(W), \atop \lambda \neq 1/R} \frac{1}{1- \lambda R} \cdot \prod_{|\lambda|= 1/R, \atop \lambda \neq 1/R} \left( 1- \lambda R \right) \\
&= \prod_{\lambda \in \text{Spec}(W), \atop \lambda \neq 1/R} \frac{1}{1- \lambda R} \cdot \prod_{a=1}^{\Delta-1} \left( 1- e^{-2 \pi i a/\Delta} \right)= C_{X} \cdot \Delta, 
\end{align*}
and thus $F= \log C_{X}+ \log \Delta$. 

\medskip

It follows from the inequality (\ref{otherfact123}) that we obtain the inequality 
\begin{align*}
\left| S_{2}(N)- F \right| &= \left| \sum_{|\lambda|< 1/R} \sum_{n> N} \frac{1}{n} (\lambda R)^{n} \right| < \frac{2 \epsilon (\rho R)^{N+1}}{1- \rho R}, 
\end{align*}
that is, 
$$
S_{2}(N)= F+ O \left( (\rho R)^{N} \right)= \log C_{X}+ \log \Delta+ O \left( (\rho R)^{N} \right). 
$$

Hence, by combining the above results, we obtain 
\begin{align*}
\sum_{n= 1}^{N} \frac{1}{n} \sum_{\lambda \in \text{Spec}(W)} (\lambda R)^{n} &= S_{1}(N)+ S_{2}(N) \\
&= \sum_{n=1}^{[ {N}/{\Delta} ]} \frac{1}{n}+ \log C_{X}+ \log \Delta+ O \left( (\rho R)^{N} \right). 
\end{align*}

\medskip

(c) \ Let $\mu(n)$ denote the M\"obius function. 
Note that $\sum_{d \mid n} |\mu(d)| \leq n$. 
It is known that 
$$
\pi(n)= \frac{1}{n} \sum_{d \mid n} \mu(d) N_{n/d}, \qquad \text{and} \qquad N_{n}= \sum_{\lambda \in \text{Spec}(W)} \lambda^{n} 
$$
(see (10.3) and (10.4) in \cite{Terras2011}, respectively). 
Combining these equalities, we obtain 
$$
n \cdot \pi(n)= \sum_{\lambda \in \text{Spec}(W)} \sum_{d \mid n} \mu(d) \lambda^{n/d}, 
$$
and therefore 
\begin{align*}
\left| n \cdot \pi(n)- \sum_{\lambda \in \text{Spec}(W)} \lambda^{n} \right| &= \left| \sum_{\lambda \in \text{Spec}(W)} \sum_{d \mid n, \atop d \geq 2} \mu(d) \lambda^{n/d} \right| \\
& \leq \sum_{\lambda \in \text{Spec}(W)} \sum_{d \mid n, \atop d \geq 2} | \mu(d)| \cdot |\lambda|^{n/d} \leq \sum_{\lambda \in \text{Spec}(W)} \sum_{d \mid n, \atop d \geq 2} | \mu(d)| \cdot |\lambda|^{n/2} \\
&< n \sum_{\lambda \in \text{Spec}(W)} \left(\frac{1}{R} \right)^{n/2} \leq 2 \epsilon n \left(\frac{1}{R} \right)^{n/2}. 
\end{align*}

On the other hand, since $0<R<1$ and $0< \alpha< 1/2$ by our assumptions, there exists a natural number $N_{0}= N_{0}(\alpha)$ such that for any $n \geq N_{0}$, 
$$
n \leq \left( \frac{1}{R} \right)^{({1}/{2}- \alpha)n}, \qquad \text{and so} \qquad n \left( \frac{1}{R} \right)^{n/2} \leq \left( \frac{1}{R} \right)^{(1- \alpha)n}. 
$$
Hence, for any $n \geq N_{0}$, 
$$
\left| n \cdot \pi(n)- \sum_{\lambda \in \text{Spec}(W)} \lambda^{n} \right| \leq 2 \epsilon \left( \frac{1}{R} \right)^{(1- \alpha)n}, 
$$
and the assertion of the item (c) follows. 

\medskip

(d) \ It is known from the equality (9) in \cite{Knuth} that 
\begin{align}\label{Knuth9}
\left| \int_{[N/\Delta]}^{\infty} \frac{P_{2k+1}(x)}{x^{2k+2}} dx \right|= O \left( \frac{1}{N^{2k+1}} \right) \quad \left(= O \left( \frac{1}{N^{k+1}} \right) \right), 
\end{align}
where $P_{2k+1}(x)$ is a periodic Bernoulli polynomial. 
Note that $[N/\Delta]= (N-a)/\Delta$. 
Recall that the $(2s-1)$-th Bernoulli numbers $B_{2s-1}$ ($s \geq 1$) are given by 
$$
B_{1}= -1/2 \qquad \text{and} \qquad B_{2s-1}=0 \quad (s \geq 2). 
$$

Then, it follows from the equality (7) in \cite{Knuth} and the above equality (\ref{Knuth9}) that 
\begin{align*}
\sum_{n=1}^{[N/\Delta]} \frac{1}{n}- \gamma & = \log \left[ \frac{N}{\Delta} \right]+ \frac{1}{2 [N/\Delta]}- \sum_{s=1}^{k} \frac{B_{2s}}{2s [N/\Delta]^{2s}}+ O \left( \frac{1}{N^{k+1}} \right) \\
& =  \log \left[ \frac{N}{\Delta} \right]- \sum_{s=1}^{2k} \frac{B_{s}}{s [N/\Delta]^{s}}+ O \left( \frac{1}{N^{k+1}} \right), 
\end{align*}
and therefore
\begin{align}\label{Knuth7}
\sum_{n=1}^{[N/\Delta]} \frac{1}{n}- \gamma & = \log \left[ \frac{N}{\Delta} \right]- \sum_{s=1}^{2k} \frac{B_{s}}{s [N/\Delta]^{s}}+ O \left( \frac{1}{N^{k+1}} \right) \nonumber \\
& = \log \left[ \frac{N}{\Delta} \right]- \sum_{s=1}^{k} \frac{B_{s}}{s [N/\Delta]^{s}}+ O \left( \frac{1}{N^{k+1}} \right) \nonumber \\
& = \log N+ \log \left( 1- \frac{a}{N} \right)- \log \Delta \nonumber \\
& \qquad \qquad - \sum_{s=1}^{k} \frac{B_{s} \Delta^{s}}{s N^{s}} \frac{1}{(1- a/N)^{s}}+ O \left( \frac{1}{N^{k+1}} \right) \nonumber \\
& = \log N- \log \Delta- \sum_{s=1}^{k} \frac{a^{s}}{s N^{s}} \nonumber \\
& \qquad \qquad - \sum_{s=1}^{k} \frac{B_{s} \Delta^{s}}{s N^{s}} \sum_{m \geq 0} \binom{s-1+m}{m} \left( \frac{a}{N} \right)^{m}+ O \left( \frac{1}{N^{k+1}} \right). 
\end{align}

On the other hand, since the inequality 
$$
\binom{s-1+m}{m}= \frac{s-1+m}{m} \cdot \frac{s-2+m}{m-1} \cdots \frac{s+1}{2} \cdot \frac{s}{1} \leq s^{m} 
$$
holds, we obtain the inequalities 
\begin{align*}
\sum_{s=1}^{k} \frac{B_{s} \Delta^{s}}{s N^{s}} \sum_{m>k-s} \binom{s-1+m}{m} \left( \frac{a}{N} \right)^{m} & \leq \sum_{s=1}^{k} \frac{B_{s} \Delta^{s}}{s N^{s}} \sum_{m> k-s} \left( \frac{sa}{N} \right)^{m} \\
&= \sum_{s=1}^{k} \frac{B_{s} \Delta^{s}}{s N^{s}} \left( \frac{sa}{N} \right)^{k-s+1} \frac{1}{1- sa/N} \\
& \leq \frac{1}{N^{k+1}} \cdot \frac{N}{N- ka} \sum_{s=1}^{k} \frac{B_{s} \Delta^{s}}{s} (sa)^{k-s+1}, 
\end{align*}
that is, 
\begin{align}\label{upperhalf}
\sum_{s=1}^{k} \frac{B_{s} \Delta^{s}}{s N^{s}} \sum_{m> k-s} \binom{s-1+m}{m} \left( \frac{a}{N} \right)^{m}= O \left( \frac{1}{N^{k+1}} \right). 
\end{align}

Hence, combining the equalities (\ref{Knuth7})(\ref{upperhalf}), we obtain 
\begin{align*}
\sum_{n=1}^{[N/\Delta]} \frac{1}{n} = \log N &- \log \Delta+ \gamma \\
& - \sum_{s=1}^{k} \left( a^{s}+ B_{s} \Delta^{s} \sum_{m=0}^{k-s} \binom{s+m-1}{m} \left( \frac{a}{N} \right)^{m} \right) \frac{1}{s N^{s}}+ O \left( \frac{1}{N^{k+1}} \right), 
\end{align*}
and the assertion of the item (d) follows after an elementary computation. 
\end{proof}

\bigskip

By using KeyLemma (c), we can prove the convergences of constants. 

\begin{lemma}\label{lem1}
Suppose that $X= (V, E)$ satisfies the same conditions as the main theorem. 

\medskip

(1) \ The series 
$$
K_{X}= \sum_{n \geq 1} \left( n \cdot \pi_{X}(n)- \sum_{\lambda \in \text{Spec}(W)} \lambda^{n} \right) R_{X}^{n}
$$
is convergent. 

\medskip

(2) \ The series 
$$
H_{X}= - \sum_{n \geq 1} \frac{1}{n} \left( n \cdot \pi_{X}(n)- \sum_{\lambda \in \text{Spec}(W)} \lambda^{n} \right) R_{X}^{n}
$$
is convergent. 
Moreover, 
$$
H_{X}= \sum_{[P]} \sum_{m \geq 2} \frac{1}{m} R_{X}^{m \ell(P)}
$$
holds. 
\end{lemma}
\begin{proof}
Let $\{a_{n}\}$ be a sequence of real numbers with $0< a_{n} \leq 1$ for any $n$. 
Note that 
\begin{align*}
\left| \sum_{n \geq N_{0}} a_{n} \left( n \cdot \pi(n)- \sum_{\lambda \in \text{Spec}(W)} \lambda^{n} \right) R^{n} \right| & \leq \sum_{n \geq N_{0}} \left| n \cdot \pi(n)- \sum_{\lambda \in \text{Spec}(W)} \lambda^{n} \right| R^{n} \\
&< 2 \epsilon \sum_{n \geq N_{0}} \left( \frac{1}{R} \right)^{(1- \alpha)n} R^{n} \\
& \leq 2 \epsilon \sum_{n \geq 1} R^{\alpha n}= \frac{2 \epsilon R^{\alpha}}{1- R^{\alpha}} 
\end{align*}
by KeyLemma (c). 
Hence, the convergences of the items (1)(2) hold from this inequality. 

\medskip

Next, we show the equality of the item (2). 
Assume that $|u|< R$. 
It follows from Fact (3) and the definition of $Z_{X}(u)$ that 
\begin{align*}
\sum_{n \geq 1} \sum_{\lambda \in \text{Spec}(W)} \frac{\lambda^{n}}{n} u^{n} &= \sum_{\lambda \in \text{Spec}(W)} \log (1- \lambda u)^{-1} \\
&= \log Z_{X}(u) = \sum_{[P]} \log (1- u^{\ell(P)})^{-1} \\
&= \sum_{[P]} \sum_{m \geq 1} \frac{1}{m} u^{m \ell(P)} = \sum_{[P]} u^{\ell(P)}+ \sum_{[P]} \sum_{m \geq 2} \frac{1}{m} u^{m \ell(P)} \\
&= \sum_{n \geq 1} \pi(n) u^{n}+ \sum_{[P]} \sum_{m \geq 2} \frac{1}{m} u^{m \ell(P)}, 
\end{align*}
and therefore 
\begin{align}\label{lem1eq}
- \sum_{n \geq 1} \frac{1}{n} \left( n \cdot \pi_{X}(n)- \sum_{\lambda \in \text{Spec}(W)} \lambda^{n} \right) u^{n}= \sum_{[P]} \sum_{m \geq 2} \frac{1}{m} u^{m \ell(P)}. 
\end{align}

On the other hand, by the graph theory prime-number theorem (see Theorem 10.1 in \cite{Terras2011}), the radius of convergence of the function 
$$
P(u)= \sum_{[P]} u^{\ell(P)}= \sum_{n \geq 1} \pi(n) u^{n}
$$
is equal to $R$. 
Note that $R^{2}< R$ since $0<R<1$. 
Then, 
\begin{align*}
\sum_{[P]} \sum_{m \geq 2} \frac{1}{m} R^{m \ell(P)} &\leq \sum_{[P]} \sum_{m \geq 2} R^{m \ell(P)}= \sum_{[P]} \frac{R^{2 \ell(P)}}{1- R^{\ell(P)}} \\
& \leq \frac{1}{1- R} \sum_{[P]} R^{2 \ell(P)} \leq \frac{1}{1-R} P(R^{2})< +\infty. 
\end{align*}

Hence, since both sides of the equality (\ref{lem1eq}) are also convergent for $u=R$, the assertion follows by the uniqueness of analytic continuation (namely, the principle of uniqueness). 
\end{proof}

\section{Proof of the main theorem}\label{sect3}

In this section, we show the main theorem. 

\bigskip

\begin{proof} \ (The main theorem) \ 
(1) \ Assume that $N$ is sufficiently large. 
Then, it follows from KeyLemma (c) that we obtain 
\begin{align*}
\left| \sum_{n=1}^{N} n \cdot \pi(n) R^{n}- \sum_{n=1}^{N} \sum_{\lambda \in \text{Spec}(W)} (\lambda R)^{n}- K \right| &= \left| \sum_{n> N} \left( n \cdot \pi(n)- \sum_{\lambda \in \text{Spec}(W)}  \lambda^{n} \right) R^{n} \right| \\
& \leq \sum_{n> N} \left| n \cdot \pi(n)- \sum_{\lambda \in \text{Spec}(W)}  \lambda^{n} \right| R^{n} \\
&< 2 \epsilon \sum_{n> N} \left( \frac{1}{R} \right)^{(1-\alpha)n} R^{n} \\
&= 2 \epsilon \sum_{n> N} R^{\alpha n}= \frac{2 \epsilon R^{\alpha (N+1)}}{1- R^{\alpha}}, 
\end{align*}
and therefore by KeyLemma (a), we obtain 
\begin{align*}
\sum_{n=1}^{N} n \cdot \pi(n) R^{n} &= \sum_{n=1}^{N} \sum_{\lambda \in \text{Spec}(W)} (\lambda R)^{n}+ K+ O \left( (\rho R)^{N} \right) \\
&= \left[ \frac{N}{\Delta} \right] \Delta+ A+ K+ O \left( (\rho R)^{N} \right)
\end{align*}
as $N \to \infty$. 
Hence, the assertion of the item (1) follows. 

\medskip

(2) \ Suppose that $N$ is sufficiently large. 
Then, it follows from KeyLemma (c) that 
\begin{align*}
\left| \sum_{n=1}^{N} \pi(n) R^{n}- \sum_{n=1}^{N} \frac{1}{n} \sum_{\lambda \in \text{Spec}(W)} (\lambda R)^{n}+ H \right| &= \left| \sum_{n> N} \left( \pi(n)- \frac{1}{n} \sum_{\lambda \in \text{Spec}(W)}  \lambda^{n} \right) R^{n} \right| \\
&= \left| \sum_{n> N} \frac{1}{n} \left( n \cdot \pi(n)- \sum_{\lambda \in \text{Spec}(W)}  \lambda^{n} \right) R^{n} \right| \\
& \leq \sum_{n> N} \left| n \cdot \pi(n)- \sum_{\lambda \in \text{Spec}(W)} \lambda^{n} \right| R^{n} \\
& < 2 \epsilon \sum_{n> N} \left( \frac{1}{R} \right)^{(1-\alpha)n} R^{n} \\
&= 2 \epsilon \sum_{n> N} R^{\alpha n}= \frac{2 \epsilon R^{\alpha (N+1)}}{1- R^{\alpha}}, 
\end{align*}
and therefore we obtain 
\begin{align*}
\sum_{n=1}^{N} \pi(n) R^{n} = \sum_{n=1}^{N} \frac{1}{n} \sum_{\lambda \in \text{Spec}(W)} (\lambda R)^{n}- H+ O \left( R^{\alpha N} \right) 
\end{align*}
as $N \to \infty$. 
Hence, it follows from KeyLemmas (b)(d) that the assertion holds. 

\medskip

(3) \ Assume that $N$ is sufficiently large, and define the following functions: 
\begin{align*}
H^{\le N} = \sum_{n \le N} \pi(n) \sum_{m \geq 2} \dfrac{1}{m} R^{mn}, \qquad \text{and} \qquad H^{> N} = \sum_{n>N} \pi(n) \sum_{m \geq 2} \dfrac{1}{m} R^{mn}.
\end{align*}
Note that $H=H^{\le N}+H^{> N}$ by Lemma \ref{lem1} (2). 
It follows from the graph theory prime-number theorem (see Theorem 10.1 in \cite{Terras2011}) that there exists a constant $c_{1}>0$ such that for any $n>N$, 
$$
\pi(n)< \frac{c_{1}}{R^{n}}. 
$$
Recall that $0< R< 1$ since $X$ is a noncycle graph. 
Then, we obtain 
\begin{align*}
H^{>N} &= \sum_{n>N} \pi(n) \sum_{m \geq 2} \dfrac{1}{m} R_{X}^{mn}< c_{1} \sum_{n> N} \sum_{m \geq 2} R^{(m-1)n} \\
&= c_{1} \sum_{n> N} \frac{R^{n}}{1- R^{n}} < \frac{c_{1}}{1- R} \sum_{n> N} R^{n}= \frac{c_{1} R}{(1-R)^{2}} R^{N}, 
\end{align*}
and therefore $H^{>N}= O(R^{N})$. 

From the item (2) and the above result, we obtain 
\begin{align*}
\sum_{n \le N} \pi(n) R^{n} +H^{\le N} &= \log N+ \gamma+ \log C_{X}- H^{>N}+ O \left( \dfrac{1}{N} \right) \\
&= \log N+ \gamma+ \log C_{X}+ O \left( \dfrac{1}{N} \right). 
\end{align*}
Since the left-hand side of the above equality is equal to 
\begin{align*}
\sum_{n \le N} \pi(n) R^{n} +H^{\le N} &= \sum_{n \le N} \pi(n) \sum_{m=1}^{\infty} \dfrac{1}{m} R^{mn} \\
&= -\sum_{n \le N} \pi(n) \log \left( 1-R^{n} \right)= -\log \left( \prod_{n \le N} \left( 1-R^{n} \right)^{\pi(n)} \right), 
\end{align*}
we obtain 
\begin{align*}
\prod_{n \le N} \left( 1-R^{n} \right)^{\pi(n)} = \dfrac{e^{-\gamma}}{C_{X}} \cdot \dfrac{1}{N} \exp \left( O \left( \dfrac{1}{N} \right) \right)= \dfrac{e^{-\gamma}}{C_{X}} \cdot \dfrac{1}{N} \left( 1+ O \left( \dfrac{1}{N} \right) \right), 
\end{align*}
and the assertion of the item (3) follows. 
\end{proof}

\begin{remark*}
(1) \ When $k=0$, our second theorem just corresponds to a special case of the second theorem due to Pollicott (see Section \ref{sect1} in this paper). 
This is proved as follows: 

In our case, the topological entropy is equal to the constant $h= -\log R> 0$. 
We now define $u= R^{s}$, $N(P)= e^{h \ell(P)}= R^{-\ell(P)}$ and $x= e^{h N}$. 
Then, the left-hand side is equal to 
$$
\sum_{n \leq N} \pi(n) R^{n}= \sum_{\ell(P) \leq N} R^{\ell(P)}= \sum_{N(P) \leq x} \frac{1}{N(P)}. 
$$
On the other hand, the right-hand side can be transformed as follows. 
Note that 
\begin{align*}
C_{X} &= - \frac{1}{R} \cdot \lim_{u \uparrow R} (u-R) Z_{X}(u)\\
&= - \frac{1}{R} \cdot \lim_{s \downarrow 1} \frac{R^{s}- R}{s-1} \cdot \lim_{s \downarrow 1} (s-1) Z_{X}(R^{s})= h \cdot \Res{s}{1} Z_{X}(R^{s}). 
\end{align*}
It follows from Lemma \ref{lem1} (2) that 
$$
H_{X}= \sum_{[P]} \sum_{n \geq 2} \frac{1}{n} R^{n \ell(P)}= \sum_{[P]} \sum_{n \geq 2} \frac{1}{n} \cdot \frac{1}{N(P)^{n}}. 
$$
By combining the above results, we obtain 
\begin{align*}
\log N &+ \gamma+ \log C_{X} - H_{X} + O \left( \frac{1}{N} \right) \\
&= \log \left( \log x \right) + \gamma+ \log \left( \Res{s}{1} Z_{X}(R^{s}) \right) - \sum_{[P]} \sum_{n \geq 2} \frac{1}{n} \cdot \frac{1}{N(P)^{n}} + O \left( \frac{1}{\log x} \right). 
\end{align*}
Hence, we obtain 
\begin{align*}
\sum_{N(P) \leq x} \frac{1}{N(P)} = \log \left( \log x \right) &+ \gamma+ \log \left( \Res{s}{1} Z_{X}(R^{s}) \right) \\
&- \sum_{[P]} \sum_{n \geq 2} \frac{1}{n} \cdot \frac{1}{N(P)^{n}} + O \left( \frac{1}{\log x} \right). 
\end{align*}

\medskip

(2) \ The error term $O(1/N)$ in our second theorem can not be replaced by $o(1/N)$ since in general, the coefficient $\Delta/2- a(N)$ of $1/N$ is not equal to zero. 
\end{remark*}



\end{document}